\documentclass[a4paper,reqno,11pt]{amsart}
\usepackage{amssymb}
\usepackage{amssymb}
\usepackage{amsfonts}
\usepackage{amsmath}
\usepackage{amsthm}
\usepackage{graphicx}
\usepackage[numbers,sort&compress]{natbib}
\usepackage{hyperref}
\usepackage[left=3.2cm,right=3.2cm,bottom=3cm,top=3cm]{geometry}
\usepackage{filecontents}

\newtheorem{theorem}{Theorem}[section]
\newtheorem*{theorem*}{Theorem}

\newtheorem{fact}[theorem]{Fact}

\newtheorem{lemma}[theorem]{Lemma}

\theoremstyle{definition}

\newtheorem*{definition*}{Definition}

\theoremstyle{remark}

\newtheorem*{remark*}{Remark}
\def \z{\mathbb{Z}}
\def \n{\mathbb{N}}
\def \BD{\operatorname{BD}}

\begin{document}
\title[High density piecewise syndeticity in groups]{High density piecewise
syndeticity of product sets in amenable groups}
\author[Di Nasso et. al.]{Mauro Di Nasso, Isaac Goldbring, Renling Jin,
Steven Leth, Martino Lupini, Karl Mahlburg}
\address{Dipartimento di Matematica, Universita' di Pisa, Largo Bruno
Pontecorvo 5, Pisa 56127, Italy}
\email{dinasso@dm.unipi.it}
\address{Department of Mathematics, Statistics, and Computer Science,
University of Illinois at Chicago, Science and Engineering Offices M/C 249,
851 S. Morgan St., Chicago, IL, 60607-7045}
\email{isaac@math.uic.edu}
\address{Department of Mathematics, College of Charleston, Charleston, SC,
29424}
\email{JinR@cofc.edu}
\address{School of Mathematical Sciences, University of Northern Colorado,
Campus Box 122, 510 20th Street, Greeley, CO 80639}
\email{Steven.Leth@unco.edu}
\address{Fakult\"{a}t f\"{u}r Mathematik, Universit\"{a}t Wien,
Oskar-Morgenstern-Platz 1, Room 02.126, 1090 Wien, Austria.}
\email{martino.lupini@univie.ac.at}
\address{Department of Mathematics, Louisiana State University, 228 Lockett
Hall, Baton Rouge, LA 70803}
\email{mahlburg@math.lsu.edu}
\thanks{The authors were supported in part by the American Institute of
Mathematics through its SQuaREs program. I. Goldbring was partially
supported by NSF CAREER grant DMS-1349399. M. Lupini was supported by the
York University Susan Mann Dissertation Scholarship and by the ERC Starting
grant no.\ 259527 of Goulnara Arzhantseva. K. Mahlburg was supported by NSF
Grant DMS-1201435.}
\subjclass[2000]{ }
\dedicatory{ }

\begin{abstract}
M. Beiglb\"ock, V. Bergelson, and A. Fish proved that if $G$ is a countable
amenable group and $A$ and $B$ are subsets of $G$ with positive Banach
density, then the product set $AB$ is piecewise syndetic. This means that
there is a finite subset $E$ of $G$ such that $EAB$ is thick, that is, $EAB$
contains translates of any finite subset of $G$. When $G=\mathbb{Z}$, this
was first proven by R. Jin. We prove a quantitative version of the
aforementioned result by providing a lower bound on the density (with
respect to a F\o lner sequence) of the set of witnesses to the thickness of $%
EAB$. When $G=\mathbb{Z}^d$, this result was first proven by the current set
of authors using completely different techniques.
\end{abstract}

\maketitle

\section{Introduction}

In the paper \cite{jin}, R. Jin proved that if $A$ and $B$ are subsets of $%
\mathbb{Z}$ with positive Banach density, then $A+B$ is \emph{piecewise
syndetic}. This means that there is $m\in \mathbb{N}$ such that $A+B+[-m,m]$
is \emph{thick}, i.e.\ it contains arbitrarily large intervals. Jin's result
has since been extended in two different ways. First, using ergodic theory,
M. Beiglb\"{o}ck, V. Bergelson, and A. Fish \cite{beiglbock_sumset_2010}
established Jin's result for arbitrary countable amenable groups (with
suitable notions of Banach density and piecewise syndeticity); in \cite%
{di_nasso_nonstandard_2014}, M. Di Nasso and M. Lupini gave a simpler proof
of the amenable group version of Jin's theorem using nonstandard analysis
which works for arbitrary (not necessarily countable) amenable groups and
also gives a bound on the size of the finite set needed to establish that
the product set is thick. Second, in \cite{di_nasso_high_2015} the current
set of authors established a \textquotedblleft quantitative
version\textquotedblright\ of Jin's theorem by proving that there is $m\in 
\mathbb{N}$ such that the set of witnesses to the thickness of $A+B+[-m,m]$
has upper density at least as large as the upper density of $A$. (We
actually prove this result for subsets of $\mathbb{Z}^{d}$ for any $d$.) The
goal of this article is to prove the quantitative version of the result of
Beiglb\"{o}ck, Bergelson, and Fish.

In the following, we assume that $G\ $is a countable amenable group\footnote{%
For those not familiar with amenable groups, let us mention in passing that
the class of (countable) amenable groups is quite robust, e.g. contains all
finite and abelian groups and is closed under subgroups, quotients,
extensions, and direct limits. It follows, for example, that every countable
virtually solvable group is amenable.}: for every finite subset $E$ of $G$
and every $\varepsilon >0$, there exists a finite subset $L$ of $G$ such
that, for every $x\in E$, we have $\left\vert xL\bigtriangleup L\right\vert
\leq \varepsilon \left\vert L\right\vert $. We recall that $G$ is amenable
if and only if it has a (left)\emph{\ F\o lner sequence}, which is a
sequence $\mathcal{S}=(S_{n})$ of finite subsets of $G$ such that, for every 
$x\in G$, we have $\frac{\left\vert xS_{n}\bigtriangleup S_{n}\right\vert }{%
|S_{n}|}\rightarrow 0$ as $n\rightarrow +\infty $. 

If $\mathcal{S}$ is a F\o lner sequence and $A\subseteq G$, we define the
corresponding upper $\mathcal{S}$\emph{-density} of $A$ to be%
\begin{equation*}
\overline{d}_{\mathcal{S}}\left( A\right) :=\limsup_{n\rightarrow +\infty }%
\frac{\left\vert A\cap S_{n}\right\vert }{\left\vert S_{n}\right\vert }.
\end{equation*}

\noindent For example, note that if $G=\mathbb{Z}^{d}$ and $\mathcal{S}%
=(S_{n})$ where $S_{n}=[-n,n]^{d}$, then $\overline{d}_{\mathcal{S}}$ is the
usual notion of upper density for subsets of $\mathbb{Z}^{d}$.

Following \cite{beiglbock_sumset_2010}, we define the (left) \emph{Banach
density of $A$}, $\mathrm{BD}\left( A\right) $, to be the supremum of $%
\overline{d}_{\mathcal{S}}\left( A\right) $ where $\mathcal{S}$ ranges over
all F\o lner sequences of $G$. One can verify (see \cite%
{beiglbock_sumset_2010}) that this notion of Banach density agrees with the
usual notion of Banach density when $G=\mathbb{Z}^{d}$.

Recall that a subset $A$ of $G$ is \emph{thick }if for every finite subset $%
L $ of $G$ there exists a right translate $Lx$ of $L$ contained in $A$. A
subset $A$ of $G$ is \emph{piecewise syndetic }if $FA$ is thick for some
finite subset $F$ of $G$.

\begin{definition*}
Suppose that $G$ is a countable discrete group, $\mathcal{S}$ is a F\o lner
sequence for $G$, $A$ is a subset of $G$, and $\alpha >0$. We say that $A$ is

\begin{itemize}
\item \emph{upper $\mathcal{S}$-thick of level $\alpha $} if for every finite
subset $L$ of $G$, the set $\left\{ x\in G:Lx\subseteq A\right\} $ has upper 
$\mathcal{S}$-density at least $\alpha $;

\item \emph{upper $\mathcal{S}$-syndetic of level $\alpha $} if there exists a
finite subset $F$ of $G$ such that $FA$ is $\mathcal{S}$-thick of level $%
\alpha $.
\end{itemize}

\end{definition*}

\noindent The following is the first main result of this paper:

\begin{theorem*}
\label{main} Suppose that $G$ is a countable amenable group and $\mathcal{S}$
is a F\o lner sequence for $G$. If $A$ and $B$ are subsets of $G$ such that $%
\overline{d}_{\mathcal{S}}(A)=\alpha >0$ and $\mathrm{BD}(B)>0$, then $BA$
is upper $\mathcal{S}$-syndetic of level $\alpha $.
\end{theorem*}

When $G=\mathbb{Z}^d$ and $\mathcal{S}=(S_n)$ where $S_n=[-n,n]^d$, we
recover \cite[Theorem 14]{di_nasso_high_2015}. We also recover \cite[Theorem
18]{di_nasso_high_2015}, which states (in the current terminology) that if $%
A,B\subseteq \mathbb{Z}^d$ are such that $\underline{d}(A)=\alpha>0$ and $%
\mathrm{BD}(B)>0$, then $A+B$ is $\mathcal{S^{\prime }}$-syndetic of level $%
\alpha$ for any subsequence $\mathcal{S}^{\prime }$ of the aforementioned $%
\mathcal{S}$.

If $\mathcal{S}$ is a F\o lner sequence for $G$, then the notions of \emph{lower $\mathcal{S}$-density $\underline{d}_\mathcal{S}$}, \emph{lower $\mathcal{S}$-thick}, and \emph{lower $\mathcal{S}$-syndetic} can be defined in the obvious ways.  In \cite[Theorem 19]{di_nasso_high_2015}, the current set of authors proved that if $A,B\subseteq \z^d$ are such that $\underline{d}_{\mathcal{S}}(A)=\alpha>0$ and $\BD(B)>0$ (where $\mathcal{S}$ is the usual F\o lner sequence for $\z^d$ as above), then for any $\epsilon>0$, $A+B$ is lower $\mathcal{S}$-syndetic of level $\alpha-\epsilon$.  (An example is also given to show that, under the previous hypotheses, $A+B$ need not be lower $\mathcal{S}$-syndetic of level $\alpha$.)  In a previous version of this paper, we asked whether or not the amenable group analog of \cite[Theorem 19]{di_nasso_high_2015} was true.  Using ergodic-theoretic methods, Michael Bj\"{o}%
rklund \cite{bjorklund} settled this question in the affirmative.  Shortly after, we realized that our techniques readily established the same result and we include our proof here. 

In their proof of the amenable group version of Jin's theorem, the authors
of \cite{di_nasso_nonstandard_2014} give a bound on the size of a finite set
needed to witness piecewise syndeticity: if $G$ is a countable amenable
group and $A$ and $B$ are subsets of $G$ of Banach densities $\alpha $ and $%
\beta $ respectively, then there is a finite subset $E$ of $G$ with $%
\left\vert E\right\vert \leq \frac{1}{\alpha \beta }$ such that $EAB$ is
thick. In Section \ref{Section:bound}, we improve upon this theorem in two
ways: we slightly improve the bound on $\left\vert E\right\vert $ from $%
\frac{1}{\alpha \beta }$ to $\frac{1}{\alpha \beta }-\frac{1}{\alpha }+1$
and we show that, if $\mathcal{S}$ is any F\o lner sequence such that $%
\overline{d}_{\mathcal{S}}\left( A\right) >0$, then $EAB$ is $\mathcal{S}$%
-thick of level $s$ for some $s>0$.


\subsection*{Notions from nonstandard analysis}

We use nonstandard analysis to prove our main results. An introduction to
nonstandard analysis with an eye towards applications to combinatorics can
be found in \cite{jin_introduction_2008}. Here, we just fix notation.

If $r,s$ are finite hyperreal numbers, we write $r\lesssim s$ to mean $%
\mathrm{st}(r)\leq \mathrm{st}(s)$ and we write $s\approx r$ to mean $%
\mathrm{st}(r)=\mathrm{st}(s)$.

If $X$ is a hyperfinite subset of ${}^\ast G$, we denote by $\mu _{X}$ the
corresponding Loeb measure. If, moreover, $Y\subseteq {}^\ast G$ is
internal, we abuse notation and write $\mu_X(Y)$ for $\mu_X(X\cap Y)$.

If $\mathcal{S}=\left( S_{n}\right) $ is a F\o lner sequence for $G$ and $%
\nu $ is an infinite hypernatural number, then we denote by $S_{\nu }$ the
value at $\nu $ of the nonstandard extension of $\mathcal{S}$. It follows
readily from the definition that $\overline{d}_{\mathcal{S}}\left( A\right) $
is the maximum of $\mu _{S_{\nu }}\left( {}^{\ast }A\right) $ as $\nu $
ranges over all infinite hypernatural numbers. It is also not difficult to
verify that a countable discrete group $G$ is amenable if and only if it
admits a \emph{F\o lner approximation}, which is a hyperfinite subset $X$ of 
${}^{\ast }G$ such that $\left\vert gX\bigtriangleup X\right\vert
/\left\vert X\right\vert $ is infinitesimal for every $g\in G$. (One
direction of this equivalence is immediate: if $(S_{n})$ is a F\o lner
sequence for $G$, then $S_{\nu }$ is a F\o lner approximation for $G$
whenever $\nu $ is an infinite hypernatural number.) 

%

\subsection{Acknowledgements}

This work was initiated during a week-long meeting at the American
Institute for Mathematics on August 4-8, 2014 as part of the SQuaRE (Structured
Quartet Research Ensemble) project \textquotedblleft Nonstandard Methods in
Number Theory.\textquotedblright\ The authors would like to thank the
Institute for the opportunity and for the Institute's hospitality during their stay.

The authors would also like to thank the anonymous referee of a previous version of this paper for helpful remarks that allowed us to sharpen our results.

\section{High density piecewise syndeticity\label{Section:high_density}}


The following fact is central to our arguments and its proof is contained in the proof of \cite[Lemma 4.6]{di_nasso_sumset_2014}.

\begin{fact}\label{goodfolner}
Suppose that $G$ is a countable amenable group and $B\subseteq G$ is such that $\BD(B)\geq \beta$.  Then there is a F\o lner approximation $Y$ of $G$ such that 
\begin{equation*}
\frac{\left\vert {}^{\ast }B\cap Y\right\vert }{\left\vert Y\right\vert }%
\gtrsim \beta
\end{equation*}%
and, for any standard $\varepsilon >0$, there exists a finite subset $F$ of $G$
such that 
\begin{equation*}
\frac{\left\vert {}^{\ast }(FB)\cap Y\right\vert }{\left\vert Y\right\vert }%
>1-\varepsilon \text{.}
\end{equation*}%
\end{fact}

For the sake of brevity, we call such a $Y$ as in the above fact a F\o lner approximation for $G$ that is \emph{good for $B$}.

\begin{lemma}\label{technicallemma}
Suppose that $G$ is a countable amenable group, $\mathcal{%
S}=\left( S_{n}\right) $ a F\o lner sequence for $G$, $Y$ a F\o lner approximation for $G$, and $A\subseteq G$.
\begin{enumerate}
\item If $\overline{d}_{\mathcal{S}}\left( A \right)\geq \alpha$, then there is $\nu>\n$ such that 
\begin{equation*}
\frac{\left\vert {}^{\ast }A\cap S_{\nu }\right\vert }{\left\vert S_{\nu
}\right\vert }\gtrsim \alpha \text{\quad and\quad }\frac{1}{\left \vert S_\nu\right \vert}\sum_{x\in S_\nu}\frac{%
\left\vert x({}^*A\cap S_\nu)^{-1}\cap Y\right\vert }{\left\vert
Y\right\vert }\gtrsim \alpha\text{.} \quad(\dagger)
\end{equation*}%
\item If $\underline{d}_{\mathcal{S}}\left( A\right)> \alpha$, then there is $\nu_0>\n$ such that $(\dagger)$ holds for all $\nu\geq \nu_0$.
\end{enumerate}
\end{lemma}

\begin{proof}
For (1), first apply transfer to the statement ``for every finite subset $E$ of $G$ and every natural number $k$, there exists $n\geq k$ such that
\begin{equation*}
\frac{1}{\left\vert S_{n}\right\vert }\left\vert A\cap S_{n}\right\vert
>\alpha -2^{-k}\text{\quad and\quad }\frac{1}{\left\vert
S_{n}\right\vert }\sum_{x\in E}\left\vert x^{-1}S_{n}\bigtriangleup
S_{n}\right\vert <2^{-k}\text{.\textquotedblright }
\end{equation*}%
Fix $K>\n$ and let $\nu$ be the result of applying the transferred statement to $Y$ and $K$.  Set $C={}^{\ast }A\cap
S_{\nu }$ and let $\chi _{C}$ denote the characteristic function of $C$. We
have%
\begin{eqnarray*}
\frac{1}{\left\vert S_{\nu }\right\vert }\sum_{x\in S_{\nu }}\frac{%
\left\vert xC^{-1}\cap Y\right\vert }{\left\vert Y\right\vert } &=&\frac{1}{%
\left\vert S_{\nu }\right\vert }\sum_{x\in S_{\nu }}\frac{1}{\left\vert
Y\right\vert }\sum_{y\in Y}\chi _{C}(y^{-1}x) \\
&=&\frac{1}{\left\vert Y\right\vert }\sum_{y\in Y}\frac{\left\vert C\cap
y^{-1}S_{\nu }\right\vert }{\left\vert S_{\nu }\right\vert } \\
&\geq &\frac{\left\vert C\right\vert }{\left\vert S_{\nu }\right\vert }%
-\sum_{y\in Y}\frac{\left\vert y^{-1}S_{\nu }\bigtriangleup S_{\nu
}\right\vert }{\left\vert S_{\nu }\right\vert } \\
&\thickapprox &\alpha \text{.}
\end{eqnarray*}%
For (2), apply transfer to the statement ``for every finite subset $E$ of $G$ and every natural number $k$, there exists $n_0\geq k$ such that, for all $n\geq n_0$, 
\begin{equation*}
\frac{1}{\left\vert S_{n}\right\vert }\left\vert A\cap S_{n}\right\vert
>\alpha -2^{-n_0}\text{\quad and\quad }\frac{1}{\left\vert
S_{n}\right\vert }\sum_{x\in E}\left\vert x^{-1}S_{n}\bigtriangleup
S_{n}\right\vert <2^{-n_0}\text{.\textquotedblright }
\end{equation*}%
Once again, fix $K>\n$ and let $\nu_0$ be the result of applying the transferred statement to $Y$ and $K$.  As above, this $\nu_0$ is as desired.
\end{proof}

\begin{theorem}
\label{Theorem:EBA}Suppose that $G$ is a countable amenable group, $\mathcal{%
S}=\left( S_{n}\right) $ a F\o lner sequence for $G$, and $A,B\subseteq G$.
If $\overline{d}_{\mathcal{S}}\left( A\right) \geq \alpha $ and $\mathrm{BD}%
(B)>0$, then $BA$ is upper $\mathcal{S}$-syndetic of level $\alpha $.
\end{theorem}

\begin{proof}
Let $Y$ be a F\o lner approximation for $G$ that is good for $B$.  Let $\nu$ be as in part (1) of Lemma \ref{technicallemma} applied to $Y$ and $A$.  Once again, set $C:={}^*A\cap S_\nu$.  
Consider the $\mu _{S_{\nu }}$-measurable function%
\begin{equation*}
f\left( x\right) =\mathrm{st}\left( \frac{\left\vert xC^{-1}\cap
Y\right\vert }{\left\vert Y\right\vert }\right) .
\end{equation*}%
Observe that 
\begin{equation*}
\int_{S_{\nu }}fd\mu _{S_{\nu }}=\mathrm{st}\left( \frac{1}{|S_{\nu }|}%
\sum_{x\in S_{\nu }}\frac{|xC^{-1}\cap Y|}{|Y|}\right) \geq \alpha ,
\end{equation*}%
whence there is some standard $r>0$ such that $\mu _{S_{\nu }}\left( \left\{ x\in S_{\nu }:f\left( x\right) \geq 2r\right\}
\right) \geq \alpha \text{.}$
%
%
%
Setting $\Gamma =\left\{ x\in S_{\nu }:\frac{|xC^{-1}\cap Y|}{|Y|}\geq r\right\}$, we have that $\mu_{S_\nu}(\Gamma)\geq \alpha$.
Take a finite subset $F$ of $G$ such that%
\begin{equation*}
\frac{\left\vert {}^{\ast }(FB)\cap Y\right\vert }{\left\vert Y\right\vert }%
>1-\frac{r}{2}.
\end{equation*}%
Fix $g\in G$.  Since $Y$ is a F\o lner approximation for $G$, we have that%
\begin{equation*}
\frac{\left\vert {}^{\ast }\left( gFB\right) \cap Y\right\vert }{\left\vert
Y\right\vert }=\frac{\left\vert {}^{\ast }\left( FB\right) \cap g^{-1}Y\right\vert }{\left\vert
Y\right\vert }\approx \frac{\left\vert {}^{\ast }\left( FB\right) \cap Y\right\vert }{\left\vert
Y\right\vert },
\end{equation*}%
whence $\frac{\left\vert {}^{\ast }\left( gFB\right) \cap Y\right\vert }{\left\vert
Y\right\vert }>1-r.$
Thus, for any $x\in \Gamma$, we have that $xC^{-1}\cap {}^*(gFB)\not=\emptyset$.
In particular, if $L$ is a finite subset of $G$, then $\Gamma \subseteq \;^{\ast }\left( \bigcap_{g\in L}gFBA\right) \text{.}$
Therefore%
\begin{equation*}
\overline{d}_{\mathcal{S}}\left( \bigcap_{g\in L}gFBA\right) \geq \mu_{S_\nu}\left({}^*( \bigcap_{g\in L}gFBA)\right)\geq \mu_{S_\nu}(\Gamma)\geq \alpha \text{.}
\end{equation*}
It follows that $BA$ is upper $\mathcal{S}$-syndetic of level $\alpha$.
\end{proof}

As mentioned in the introduction, Theorem 14 and Theorem 18 of \cite%
{di_nasso_high_2015} are immediate consequences of Theorem \ref{Theorem:EBA}%
, after observing that the sequence of sets $\left[ -n,n\right] ^{d}$ as
well as any of its subsequences is a F\o lner sequence for $\mathbb{Z}^{d}$.
Example 15 of \cite{di_nasso_high_2015} shows that the conclusion in Theorem %
\ref{Theorem:EBA} is optimal, even when $G$ is the additive group of
integers and $\mathcal{S}$ is the F\o lner sequence of intervals $\left[ 1,n%
\right] $.

\begin{theorem}
\label{Theorem:EBA-lower}Suppose that $G$ is a countable amenable group, $%
\mathcal{S}=\left( S_{n}\right) $ a F\o lner sequence for $G$, and $%
A,B\subseteq G$. If $\underline{d}_{\mathcal{S}}\left( A\right) \geq \alpha $
and $\mathrm{BD}(B)>0$, then $BA$ is lower $\mathcal{S}$-thick of level $%
\alpha -\varepsilon $ for every $\varepsilon >0$.
\end{theorem}

\begin{proof}
Without loss of generality, we may suppose that $\underline{d}_{\mathcal{S}}(A)>\alpha$.  Fix a F\o lner approximation $Y$ for $G$ that is good for $B$ and $\nu _{0}>\n$ as in part (2) of Lemma \ref{technicallemma} applied to $Y$ and $A$.  Fix $\nu \geq \nu _{0}$ and standard $\varepsilon >0$ with $\varepsilon
<\alpha $.
Set 
\begin{equation*}
\Gamma _{\nu }=\left\{ x\in S_{\nu }:\frac{\left\vert xC^{-1}\cap
Y\right\vert }{\left\vert Y\right\vert }\geq \varepsilon \right\}
\end{equation*}%
and observe that $\frac{\left\vert \Gamma _{\nu }\right\vert }{\left\vert S_{\nu }\right\vert }%
>\alpha -\varepsilon \text{.}$
Fix a finite subset $F$ of $G$ such that%
\begin{equation*}
\frac{\left\vert {}^{\ast }(FB)\cap Y\right\vert }{\left\vert Y\right\vert }%
>1-\frac{\varepsilon}{2} \text{.}
\end{equation*}%
Fix $g\in G$.  Since $Y$ is a F\o lner approximation of $G$, arguing as in the proof of the previous theorem, we conclude that%
\begin{equation*}
\frac{\left\vert {}^{\ast }\left( gFB\right) \cap Y\right\vert }{\left\vert
Y\right\vert }>1-\varepsilon \text{.}
\end{equation*}%
Fix $L\subseteq G$ finite.  Once again, it follows that $\Gamma _{\nu }\subset \;^{\ast }\left( \bigcap_{g\in L}gFBA\right)$
whence
\begin{equation*}
\frac{\left\vert {}^{\ast }\left( \bigcap_{g\in L}gFBA\right) \cap S_{\nu
}\right\vert }{\left\vert S_{\nu }\right\vert }\geq \frac{\left\vert \Gamma
_{\nu }\right\vert }{\left\vert S_{\nu }\right\vert }> \alpha
-\varepsilon \text{.}
\end{equation*}%
Since the previous inequality held for every $\nu \geq \nu _{0}$, by transfer we can conclude that%
\begin{equation*}
\underline{d}_{\mathcal{S}}\left( \bigcap_{g\in L}gFBA\right) \geq \alpha
-\varepsilon \text{.}
\end{equation*}
It follows that $BA$ is lower $\mathcal{S}$-syndetic of level $\alpha-\epsilon$.
\end{proof}

As mentioned in the introduction, Theorem \ref{Theorem:EBA-lower} is a generalization of \cite[Theorem 19]{di_nasso_high_2015} and was first proven by M. Bj\"{o}%
rklund using ergodic-theoretic methods.

\section{A bound on the number of translates\label{Section:bound}}

The following theorem is a refinement of \cite[Corollary 3.4]%
{di_nasso_nonstandard_2014}. In particular, we improve the bound on the
number of translates, and also obtain an estimate on the $\mathcal{S}$%
-density of translates that witness the thickness of $EBA$.

\begin{theorem}
Suppose that $G$ is a countable amenable group, $\mathcal{S}=\left(
S_{n}\right) $ a F\o lner sequence for $G$, and $A,B\subseteq G$. If $d_{%
\mathcal{S}}\left( A\right) \geq \alpha $ and $\mathrm{BD}\left( B\right)
\geq \beta $, then there exists $s>0$ and a finite subset $E\subseteq G$
such that $\left\vert E\right\vert \leq \frac{1}{\alpha \beta }-\frac{1}{%
\alpha }+1$ and $EBA$ is $\mathcal{S}$-thick of level $s$.
\end{theorem}

\begin{proof}
Reasoning as in the proof of Theorem \ref{Theorem:EBA} one can
prove that there exist $\nu \in {}^{\ast }\mathbb{N}\backslash \mathbb{N}$,
a F\o lner approximation $Y$ of $G$, a standard $r>0$, and an infinitesimal $%
\eta \in {}^{\ast }\mathbb{R}_{+}$ such that, if $C={}^{\ast }A\cap S_{\nu }$%
, $D={}^{\ast }B\cap Y$, and%
\begin{equation*}
\Gamma :=\left\{ x\in S_{\nu }:\frac{1}{\left\vert Y\right\vert }\left\vert
xC^{-1}\cap D\right\vert \geq \alpha \beta -\eta \right\} \text{,}
\end{equation*}%
then $\left\vert \Gamma \right\vert \geq r\left\vert S_{\nu }\right\vert $.
Fix a family $\left( p_{g}\right) _{g\in G}$ of strictly positive standard
real numbers such that $\sum_{g\in G}p_{g}\leq \frac{1}{2}$.

We now define a sequence of subsets $(H_{n})$ of $G$ and a sequence $(s_{n})$
from $G$. Define%
\begin{equation*}
H_{0}:=\left\{ g\in G:\frac{1}{\left\vert \Gamma \right\vert }\left\vert
\left\{ x\in \Gamma :gx\notin {}^{\ast }\left( BA\right) \right\}
\right\vert >p_{g}\right\} \text{.}
\end{equation*}%
If $H_{n}$ has been defined and is nonempty, let $s_{n}$ be any element of $%
H_{n}$ and set%
\begin{equation*}
H_{n+1}:=\left\{ g\in G:\frac{1}{\left\vert \Gamma \right\vert }\left\vert
\left\{ x\in \Gamma :gx\notin {}^{\ast }\left( \left\{ s_{0},\ldots
,s_{n}\right\} BA\right) \right\} \right\vert >p_{g}\right\} \text{.}
\end{equation*}%
If $H_{n}=\varnothing $ then we set $H_{n+1}=\varnothing $.

We claim $H_{n}=\emptyset $ for $n>\left\lfloor \frac{1}{\alpha \beta }-%
\frac{1}{\alpha }\right\rfloor $. Towards this end, suppose $H_{n}\neq
\varnothing $. For $0\leq k\leq n$, take $\gamma _{k}\in \Gamma $ such that%
\begin{equation*}
s_{k}\gamma _{k}\notin {}^{\ast }\left( \left\{ s_{0},\ldots
,s_{k-1}\right\} BA\right) .
\end{equation*}%
Observe that the sets $s_{0}D$, $s_{1}(\gamma _{1}C^{-1}\cap D)$, ..., $%
s_{n}(\gamma _{n}C^{-1}\cap D)$ are pairwise disjoint. In fact, if 
\begin{equation*}
s_{i}D\cap s_{j}\gamma _{j}C^{-1}\neq \varnothing
\end{equation*}%
for $0\leq i<j\leq n$, then%
\begin{equation*}
s_{j}\gamma _{j}\in s_{i}DC\subseteq {}^{\ast }\left( \left\{ s_{0},\ldots
,s_{j-1}\right\} BA\right) \text{,}
\end{equation*}%
contradicting the choice of $\gamma _{j}$. Therefore we have that%
\begin{eqnarray*}
1 &\gtrsim &\frac{1}{\left\vert Y\right\vert }\left\vert s_{0}D\cup
s_{1}\left( (\gamma _{1}C^{-1}\cap D\right) \cup \cdots \cup s_{n}\left(
(\gamma _{n}C^{-1}\cap D\right) \right\vert \\
&\geq &\frac{1}{\left\vert Y\right\vert }\left( \left\vert D\right\vert
+\sum_{i=1}^{n}\left\vert \gamma _{i}C^{-1}\cap D\right\vert \right) \\
&\geq &\beta +\alpha \beta n\text{.}
\end{eqnarray*}%
It follows that $n\leq \left\lfloor \frac{1}{\alpha \beta }-\frac{1}{\alpha }%
\right\rfloor $.

Take the least $n$ such that $H_{n}=\emptyset $. Note that $n\leq
\left\lfloor \frac{1}{\alpha \beta }-\frac{1}{\alpha }\right\rfloor +1$. If $%
n=0$ then $BA$ is already $\mathcal{S}$-thick of level $r$, and there is
nothing to prove. Let us assume that $n\geq 1$, and set $E=\left\{
s_{0},\ldots ,s_{n-1}\right\} $. It follows that, for every $g\in G$, we
have that%
\begin{equation*}
\frac{1}{\left\vert \Gamma \right\vert }\left\vert \left\{ x\in \Gamma
:gx\in {}^{\ast }\left( EBA\right) \right\} \right\vert \geq 1-p_{g}\text{.}
\end{equation*}%
Suppose that $L$ is a finite subset of $G$. Then%
\begin{equation*}
\frac{1}{\left\vert \Gamma \right\vert }\left\vert \left\{ x\in \Gamma
:Lx\subseteq {}^{\ast }\left( EBA\right) \right\} \right\vert \geq
1-\sum_{g\in L}p_{g}\geq 1-\sum_{g\in G}p_{g}\geq \frac{1}{2}\text{.}
\end{equation*}%
Therefore%
\begin{equation*}
\frac{1}{\left\vert S_{\nu }\right\vert }\left\vert \left\{ x\in S_{\nu
}:Lx\subseteq {}^{\ast }\left( EBA\right) \right\} \right\vert \geq \frac{r}{%
2}\text{.}
\end{equation*}%
This shows that $EBA$ is $\mathcal{S}$-thick of level $\frac{r}{2}$.
\end{proof}

With a similar argument and using Markov's inequality \cite[Lemma 1.3.15]%
{tao_introduction_2011} one can also prove the following result. We omit the
details.

\begin{theorem}
Suppose that $G$ is a countable amenable group, $\mathcal{S}=\left(
S_{n}\right) $ a F\o lner sequence for $G$, and $A,B\subseteq G$. If $d_{%
\mathcal{S}}\left( A\right) \geq \alpha $ and $\mathrm{BD}\left( B\right)
\geq \beta $, then for every $\gamma \in (0,\alpha \beta ]$ there exists $%
E\subseteq G$ such that $\left\vert E\right\vert \leq \frac{1-\beta }{\gamma 
}+1$ and $EBA$ is $\mathcal{S}$-syndetic of level $\frac{\alpha \beta
-\gamma }{1-\gamma }$.
\end{theorem}

%

\bibliographystyle{amsplain}
\bibliography{bibliography}

\end{document}